\documentclass[12pt]{amsart}
\usepackage{amssymb,amscd,amsthm,amsmath,color,hyperref}
\usepackage{fullpage}
\newcommand{\PP}{\mathbb{P}}
\newcommand{\OO}{\mathcal{O}}
\newcommand{\II}{\mathcal{I}}

\newcommand{\Hom}{\textsf{Hom}}

\newcommand{\End}{\operatorname{End}}

\newcommand{\sHom}{\mathcal{H}\text{om}}
\newcommand{\Ext}{\operatorname{Ext}}

\theoremstyle{plain}
\newtheorem{lemma}{Lemma}[section]
\newtheorem*{theorem*}{Theorem}
\newtheorem*{lemma*}{Lemma}
\newtheorem*{proposition*}{Proposition}
\newtheorem*{conjecture*}{Conjecture}
\newtheorem*{corollary*}{Corollary}
\newtheorem*{problem*}{Problem}
\newtheorem{theorem}[lemma]{Theorem}
\newtheorem{conjecture}[lemma]{Conjecture}
\newtheorem{corollary}[lemma]{Corollary}
\newtheorem{proposition}[lemma]{Proposition}

\theoremstyle{definition}
\newtheorem{definition}[lemma]{Definition}

\begin{document}

\title{Normal bundles of rational curves on complete intersections}
\author[I. Coskun]{Izzet Coskun}
\address{Department of Mathematics, Stat. and CS \\University of Illinois at Chicago, Chicago, IL 60607}
\email{coskun@math.uic.edu}

\author[E. Riedl]{Eric Riedl}
\email{ebriedl@uic.edu}

\thanks{During the preparation of this article the first  author was partially supported by the NSF grant DMS 1500031; and the second author was partially supported by an NSF RTG grant DMS-1246844}

\begin{abstract}
Let $X \subset \PP^n$ be a general Fano complete intersection of type $(d_1,\dots, d_k)$. If at least one $d_i$  is greater than $2$, we show that $X$ contains rational curves of degree $e \leq n$ with balanced normal bundle. If all $d_i$ are $2$ and $n\geq 2k+1$, we show that $X$ contains rational curves of degree $e \leq n-1$ with balanced normal bundle. As an application, we prove a stronger version of the theorem of Z. Tian \cite{Tian},  Q. Chen and Y. Zhu \cite{ChenZhu} that $X$ is separably rationally connected by exhibiting very free rational curves in $X$ of optimal degrees.
\end{abstract}

\maketitle

\section{Introduction}
Spaces of rational curves play a fundamental role in studying the geometry and arithmetic of projective varieties. The local structure of these spaces at a point $[C]$ is controlled by the normal bundle of $C$. Normal bundles of rational curves have been studied extensively (see \cite{AlzatiRe, CoskunRiedl, EisenbudVandeven, EisenbudVandeven2, Ran, Sacchiero2, Sacchiero} for rational curves in $\PP^n$ and  \cite{Debarre, Furukawa, Larson, Kollar} for rational curves in more general varieties). In this paper, we compute the normal bundles of  rational normal curves in general complete intersections and show that under mild assumptions they are balanced.  As a corollary, we obtain a stronger quantitative version of  the theorem of Z. Tian \cite{Tian},  Q. Chen and Y. Zhu \cite{ChenZhu} that a general Fano complete intersection is separably rationally connected. We provide optimal bounds on the degree of very free rational curves on general complete intersections. We work over an algebraically closed field $K$ of arbitrary characteristic $p$.  

Let $2 \leq d_1 \leq d_2 \leq \cdots \leq d_k$ be a sequence of positive integers and let $d=\sum_{i=1}^r d_i$. Let $X$ be a general Fano complete intersection of type $(d_1,\dots,d_k)$ in $\PP^n$. The normal bundle of a rational curve $C \subset X$  splits as a direct sum of line bundles $N_{C/X}= \bigoplus_{i=1}^{n-k-1} \OO_{\PP^1}(a_i)$. We say that the normal bundle $N_{C/X}$ is {\em balanced} if $|a_i - a_j| \leq 1$ for every $i,j$.
The main theorem of this paper is the following.

\begin{theorem*}
Let $X \subset \PP^n$ be a general Fano complete intersection of type $(d_1, \dots, d_k)$. 
\begin{enumerate}
\item If $d_k \geq 3$, then $X$ contains rational curves with balanced normal bundle of every degree $1 \leq e \leq n$.
\item If $d_1 = \cdots = d_k =2$ and $n\geq 2k+1$, then $X$ contains  rational curves with balanced normal bundle of every degree $1 \leq e \leq n-1$.
\end{enumerate}
\end{theorem*}

In fact, we prove the existence of rational curves with balanced normal bundle in a slightly larger range of degrees (see Corollary \ref{cor-nPlusOneCase}, Theorem \ref{Thm-mainbalanced} and Propositions \ref{4n+1}, \ref{glue2k} and \ref{gluequadrics}).  As an application, we obtain a simple new proof of a theorem of Z. Tian \cite{Tian}, Q. Chen and Y. Zhu \cite{ChenZhu} that a general Fano complete intersection is separably rationally connected. Our method also provides sharp bounds on the degree of very free rational curves on general Fano hypersurfaces and complete intersections. Recall that a projective variety $X$ is {\em separably rationally connected} if there exists a variety $M$ and a rational map $\pi: \PP^1 \times M \dashrightarrow  X$ such that the evaluation map $ev: \PP^1 \times \PP^1 \times M \dashrightarrow X \times X$ is dominant and separable. A rational curve $f: \PP^1 \rightarrow X$ is called {\em free} if $f^* T_X$ is globally generated. A rational curve $f: \PP^1 \rightarrow X$ on a projective variety $X$ is called {\em very free} if $f^* T_X \otimes \OO_{\PP^1}(-1)$ is globally generated, equivalently if every summand in $f^* T_X$ has degree at least $1$. If $X$ contains a very free rational curve, then $X$ is separably rationally connected \cite[Corollary 4.17]{Debarre}. Twisting the standard exact sequence  
$$0 \longrightarrow T_{\PP^1} = \OO_{\PP^1}(2) \longrightarrow f^* T_X \longrightarrow N_f \longrightarrow 0$$ by $\OO_{\PP^1}(-2)$ and taking $H^1$, we conclude that if every factor in $N_f$ has degree at least $1$, then the image of $f$ is a very free rational curve on $X$. 

\begin{corollary*}
Let $X$ be a general Fano complete intersection in $\PP^n$  of type $(d_1, \dots, d_k)$. Then $X$ contains a very free rational curve of degree $e \geq m = \lceil \frac{n-k+1}{n-d+1}\rceil$.
\end{corollary*}
By degree considerations, $X$ cannot have a very free rational curve of degree less than $m$. Hence, the corollary is optimal. The case of hypersurfaces of degree $n$ in $\PP^n$ was already known by the work of Zhu \cite{Zhu}.

The positivity properties of normal bundles of rational curves on Fano hypersurfaces in $\PP^n$ can be subtle in positive characteristic. For instance, although a general Fano hypersurface $X$  of degree $d < n$ contains a free line, for any degree $e$ there exist examples of Fano hypersurfaces containing no free rational curves of degree less than $e$ \cite{Conduche, Shen, Bridges}. Nevertheless, it is conjectured that smooth Fano hypersurfaces always contain free rational curves \cite{Bridges}, although this remains open. Tian  \cite{Tian} shows that if $X$ contains a free rational curve, then $X$  also contains a very free rational curve. For Fano threefolds of Picard rank one, the existence of free rational curves  follows by results of Shen \cite{Shen2}.

Our results hinge on the following analysis.  If $C \subset X$ with $C$ smooth, we obtain a sequence
\[ 0 \to N_{C/X} \to N_{C/\PP^n} \to N_{X/\PP^n}|_C = \OO_{\PP^1}(ed) .\]
If we let $X$ vary among complete intersections of type $(d_1, \dots, d_k)$, then we obtain a map $\phi: \bigoplus_i H^0(\II_{C/\PP^n}(d_i)) \to \bigoplus_i \Hom(N_{C/\PP^n}, \OO(ed_i)).$ Our main technical theorem is the following.

\begin{theorem*}
Let $(d_1, \dots, d_k)$ be a tuple of integers with each $d_i \geq 3$ and let $C$ be a general rational curve in $\PP^n$ of degree $e \leq n$. Then $\phi$ is surjective.
\end{theorem*}
Unfortunately, when some $d_i=2$, the map $\phi$ is not surjective. This forces us to analyze complete intersections of quadric hypersurfaces separately.

\subsection*{Organization of the paper} After discussing some preliminaries in \S \ref{sec-preliminaries}, in \S \ref{sec-3} we study the map $\phi$ and prove our main technical theorem. The main theorem follows when $d_i \geq 3$ for all $i$. In \S \ref{sec-quadrics}, we compute normal bundles of low degree rational curves in complete intersections of quadrics. In \S \ref{sec-final}, using a degeneration argument we prove our main theorem in the general case.

\subsection*{Acknowledgements:} We thank Roya Beheshti, Lawrence Ein, Joe Harris, Jason Starr, Qile Chen and Yi Zhu for useful discussions.

\section{Preliminaries}\label{sec-preliminaries}
In this section, we collect well-known but useful facts on normal bundles of rational curves in $\PP^n$. We refer the reader to \cite{CoskunRiedl, EisenbudVandeven, GhioneSacchiero} for more details.

\subsection*{Vector bundles on $\PP^1$} Every vector bundle $V$ on $\PP^1$ is a direct sum of line bundles $V \cong \bigoplus_{i=1}^r \OO_{\PP^1}(a_i)$. The sequence of integers $a_1, \dots, a_r$ is called the {\em splitting type} of $V$ and $V$ is called {\em balanced} if $|a_i - a_j| \leq 1$ for every $i,j$. In this paper, we  study the splitting type of the normal bundle of a rational curve in a complete intersection. The following lemma will be crucial in determining when such a bundle is balanced. Let $\delta_{i,j}$ denote the Kr\"{o}necker delta function.

\begin{lemma}\label{lem-generalkernel}
Let $V= \bigoplus_{i=1}^r \OO_{\PP^1}(a_i)$ be  a vector bundle of rank $r \geq 2$. Let $\phi \in \Hom(V, \OO_{\PP^1}(d))$ be a general homomorphism. 
\begin{enumerate}
\item If $\frac{\sum_{i=1}^r a_i -d}{r-1} \leq \min_i \{a_i \}, $ then the  kernel of $\phi$ is balanced.
\item Moreover, if $V= \OO_{\PP^1}(a)^m \oplus \OO_{\PP^1}(a+1)^{r-m}$ is balanced and $d \geq a + \delta_{0,m}$, then the kernel of $\phi$ is  balanced.
\end{enumerate}
\end{lemma}

\begin{proof}
 In a family of vector bundles on $\PP^1$ with fixed rank and degree, being balanced is an open condition. Hence, it suffices to exhibit one $\phi$ with balanced kernel. To prove (1), let $W$ be the balanced vector bundle of rank $r-1$ and degree $\sum_{i=1}^r a_i -d$. By our numerical assumption, $\sHom(W, V)$ is globally generated. Hence, the general $\psi$ in $\Hom(W, V)$ is an inclusion of vector bundles \cite[Lemma 2.6]{Huizenga}. Consequently, the cokernel is the line bundle $\OO_{\PP^1}(d)$. The natural map from $V$ to the cokernel of $\psi$ provides the desired homomorphism $\phi$. 
 
 The proof of (2) is similar. If $V= \OO_{\PP^1}(a)^m \oplus \OO_{\PP^1}(a+1)^{r-m}$ and $d \geq a+r-m$,  (1) implies  (2). On the other hand, if $a+ \delta_{0,m} \leq d < a + r-m$, then there is still an injective bundle map $\psi: \OO_{\PP^1}(a)^{m+d-a-1} \oplus \OO_{\PP^1}(a+1)^{r-m-d+a} \rightarrow V$, whose cokernel is $\OO_{\PP^1}(d)$. Hence, the kernel of a general $\phi \in \Hom(V, \OO_{\PP^1}(d))$ is balanced even when $a + \delta_{0,m} \leq d < a + r-m$.
\end{proof}

\begin{corollary}
\label{cor-generalkernelnormalbundle}
If $C$ is a smooth degree $e$ rational curve in $\PP^n$ and $\phi \in \Hom(N_{C/\PP^n}, \OO_{\PP^1}(ed))$ is a general element with $d \geq 2$, then the kernel of $\phi$ is balanced.
\end{corollary}
\begin{proof}
Let $N_{C/\PP^n} = \bigoplus_{i=1}^{n-1} \OO(a_i)$.  The sequence
$0 \to T_C \to T_{\PP^n}|_C \to N_{C/\PP^n} \to 0$
shows that $\sum_{i=1}^{n-1} a_i = e(n+1)-2$.  By the Euler sequence for $\PP^n$, each factor of $T_{\PP^n}|_C$ has degree at least $e$. Therefore, each factor of $N_{C/\PP^N}$ must have degree at least $e$, so $\min_i \{a_i \} \geq e$. Hence,
\[ \frac{ \sum_{i=1}^{n-1} a_i - ed}{n-1} = \frac{e(n+1)-2-ed}{n-1} < \frac{e(n-1)}{n-1} = e \leq \min_i \{a_i \} . \]
The result follows from Lemma \ref{lem-generalkernel}.
\end{proof}

An argument similar to that of Lemma \ref{lem-generalkernel} yields the following lemma. 

\begin{lemma}\label{lem-generalextension}
Let $a_1 \leq \cdots \leq a_j < a_{j+1} = \cdots = a_r$ be a sequence of integers. If $$\sum_{i=1}^j (a_r -1 -a_i) \leq d -a_r +1,$$ then a general extension of $\OO_{\PP^1}(d)$ by $\bigoplus_{i=1}^r \OO_{\PP^1}(a_i)$ is balanced.
\end{lemma}

\subsection*{Normal bundles of rational normal curves} Assume that $C$ is a smooth rational curve of degree $e$ parameterized by $f: \PP^1 \rightarrow \PP^n$. If $p \nmid e$, we can use the Euler sequences on $\PP^1$ and $\PP^n$ and the Jacobian $J f$ to compute $N_{C/\PP^n}$. There is a commutative diagram
\catcode`\@=8
\newdimen\cdsep
\cdsep=2em

\def\cdstrut{\vrule height .2\cdsep width 0pt depth .1\cdsep}
\def\@cdstrut{{\advance\cdsep by 2em\cdstrut}}

\def\arrow#1#2{
  \ifx d#1
    \llap{$\scriptstyle#2$}\left\downarrow\cdstrut\right.\@cdstrut\fi
  \ifx u#1
    \llap{$\scriptstyle#2$}\left\uparrow\cdstrut\right.\@cdstrut\fi
  \ifx r#1
    \mathop{\hbox to \cdsep{\rightarrowfill}}\limits^{#2}\fi
  \ifx l#1
    \mathop{\hbox to \cdsep{\leftarrowfill}}\limits^{#2}\fi
}
\catcode`\@=10

\cdsep=2em
$$
\begin{matrix}
 & & & & 0 & & 0 & & \cr
 & & & & \arrow{d}{} & & \arrow{d}{} & & \cr
0 & \arrow{r}{} & \OO_{\PP^1}  & \arrow{r}{} & \OO_{\PP^1}(1)^2 & \arrow{r}{} & T_{\PP^1} & \arrow{r}{} & 0          \cr
& &  \arrow{d}{e} & & \arrow{d}{J f} & & \arrow{d}{df} \cr
0 & \arrow{r}{} & \OO_{\PP^1} & \arrow{r}{} & \OO_{\PP^n}(1)^{n+1}|_C & \arrow{r}{} & T_{\PP^n}|_C & \arrow{r}{} & 0 \cr
& & & & \arrow{d}{} & & \arrow{d}{} \cr
& & & & N_{C/\PP^n} & \arrow{r}{=} & N_{C/\PP^n} \cr
& & & & \arrow{d}{} & & \arrow{d}{} \cr
& & & & 0 & & 0 \cr
\end{matrix}
$$
where the rows are given by the Euler sequences for $\PP^1$ and $\PP^n$, respectively, and the commutativity of the left square is Euler's relation (see \cite{EisenbudVandeven, GhioneSacchiero}). It is important for this sequence that $p \nmid e$. 

For $e \leq n$, a rational normal curve $R_e$ of degree $e$ in $\PP^n$ is projectively equivalent to the curve defined by
$$[s^e : s^{e-1} t : s^{e-2} t : \cdots : s t^{e-1} : t^{e}: 0 : \cdots : 0] : \PP^1 \rightarrow \PP^n.$$ If $p \nmid e$,  Theorem \cite[Theorem 3.2]{CoskunRiedl} implies that $N_{R_e/\PP^n} \cong \OO_{\PP^1}(e+2)^{e-1} \oplus \OO_{\PP^1}(e)^{n-e}$. However, this method does not work when $p | e$, so we adopt a different approach here.

Let $R_n$ be the rational normal curve of degree $n$ in $\PP^n$. The ideal of $R_n$ is cut out by the $2 \times 2$ minors of the matrix
\[ \left( \begin{array}{cccc}
x_1 & x_2 & \dots & x_n \\
x_0 & x_1 & \dots & x_{n-1}
\end{array} \right). \]
Let $q_{ij} = x_ix_{j-1} - x_j x_{i-1}$. Then the relations between the $q_{ij}$ are obtained from the $3 \times 3$ minors of the matrices
\[ \left( \begin{array}{cccc}
x_1 & x_2 & \dots & x_n \\
x_1 & x_2 & \dots & x_n \\
x_0 & x_1 & \dots & x_{n-1}
\end{array} \right) \quad \mbox{or}  \quad
 \left( \begin{array}{cccc}
x_0 & x_1 & \dots & x_{n-1} \\
x_1 & x_2 & \dots & x_n \\
x_0 & x_1 & \dots & x_{n-1}
\end{array} \right) .\]

An element $\alpha \in H^0(N_{R_n/\PP^n}) = \Hom(\II_{R_n/\PP^n},\OO_C)$ is determined by $\alpha(q_{ij})$, the image of the generators of the ideal of $R_n$. In fact, the next proposition shows that it suffices to know $\alpha (q_{i,i+1})$ for $1 \leq i \leq n-1$. 

\begin{proposition}
\label{RNCnormalBundle}
An element $\alpha \in H^0(N_{R_n/\PP^n})$ is determined by the images $\alpha(q_{i,i+1})$, for $1 \leq i \leq n-1$. Furthermore,  $s^{n-i-1}t^{i-1}$ divides $\alpha(q_{i,i+1})$ and this is the only constraint on  $\alpha(q_{i,i+1})$. If  $b_{i,i+1}$, for $1 \leq i \leq n-1$, are arbitrary polynomials of degree $n+2$, there exists an element $\alpha \in  H^0(N_{R_n/\PP^n})$ such that $\alpha(q_{i,i+1}) = s^{n-i-1}t^{i-1}b_{i,i+1}$. 
\end{proposition}
\begin{proof}
Let $\alpha \in H^0(N_{R_n/\PP^n})$. We start by proving that $s^{n-i-1}t^{i-1}$ divides $\alpha(q_{i,i+1})$. Apply $\alpha$ to the relation
\[ x_1q_{i,i+1} - x_i q_{1,i+1} + x_{i+1} q_{1,i} = 0 \]
to obtain
\[ s^{n-1}t \alpha(q_{i,i+1}) - s^{n-i}t^i \alpha(q_{1,i+1}) + s^{n-i-1}t^{i+1} \alpha(q_{1,i})=0 . \]
Hence, $t^{i-1}$ divides $\alpha(q_{i,i+1})$. Similarly,  apply $\alpha$ to the relation $$x_i q_{i+1,n} - x_{i+1}q_{i,n}+x_n q_{i,i+1} = 0$$ to obtain
\[ s^{n-i}t^i \alpha(q_{i,i+1}) - s^{n-i-1}t^{i+1} \alpha(q_{i,n}) + t^n \alpha(q_{i,i+1}) = 0, \]
which implies $s^{n-i-1}$ divides $\alpha(q_{i,i+1})$. Thus, $s^{n-i-1}t^{i-1}$ divides $\alpha(q_{i,i+1})$ and we may define polynomials  $$b_{i,i+1} := \frac{\alpha(q_{i,i+1})}{s^{n-i-1}t^{i-1}}.$$

We next show that 
\[ \alpha(q_{i,k}) = \sum_{\ell=i}^{k-1} s^{n-k-i+\ell}t^{k+i-\ell-2} b_{\ell,\ell+1} \]
by induction on $k-i$. The base case $k-i = 1$ is simply the definition of $b_{i,i+1}$. Suppose the formula holds for $\alpha (q_{r,m})$ with $r-m < i-k$. We prove the formula for $\alpha(q_{i,k})$. Apply $\alpha$ to the relation $$x_i q_{i+1,k} - x_{i+1}q_{i,k} + x_k q_{i,i+1} = 0$$ to obtain
\[ s^{n-i}t^i \alpha(q_{i+1,k}) - s^{n-i-1}t^{i+1} \alpha(q_{i,k}) + s^{n-k}t^k \alpha(q_{i,i+1}) = 0 .\]
Using our induction hypothesis and solving for $\alpha(q_{i,k})$, we obtain
\[ \alpha(q_{i,k}) = \sum_{\ell=i}^{k-1} s^{n-k-i+\ell}t^{k+i-\ell-2} b_{\ell, \ell+1}\] as desired. Therefore, the images $\alpha(q_{i,i+1})$ for $1 \leq i \leq n-1$ determine $\alpha$.

It remains to show that $b_{i,i+1}$ can be chosen arbitrarily. We need to check that $\alpha(q_{i,j})$ satisfies the relations among $q_{i,j}$. Set $u=2n-k-i-j$ and $v=i+j+k$. Apply $\alpha$ to $$x_i q_{j,k} - x_j q_{i,k} + x_k q_{j,k}$$ to obtain
\[  \sum_{\ell = j}^{k-1} s^{u+\ell}t^{v-\ell-2}b_{\ell, \ell+1} - \sum_{\ell = i}^{k-1} s^{u+\ell}t^{v-\ell-2} b_{\ell, \ell+1} +\sum_{\ell=i}^{j-1} s^{u+\ell}t^{v-\ell-2} b_{\ell, \ell+1} = 0 . \] Hence, the relation holds. A similar calculation works for the relation $$x_{i-1} q_{j,k} - x_{j-1} q_{i,k} + x_{k-1} q_{j,k}=0.$$ This completes the proof that $b_{i, i+1}$ can be chosen arbitrarily.
\end{proof}

\begin{corollary}
\label{cor-RnNormalBundle}
The normal bundle of $R_n$ in $\PP^n$ is $N_{R_n/\PP^n} = \OO_{\PP^1}(n+2)^{n-2}$.
\end{corollary}

\begin{corollary}
For an integer $e \leq n$, the normal bundle $N_{R_e/\PP^n}$ is $\OO_{\PP^1}(e+2)^{e-1} \oplus \OO_{\PP^1}(e)^{n-e}$.
\end{corollary}

\begin{proof}
Let $\Lambda$ be the $e$-plane spanned by $R_e$. We have the short exact sequence 
\[ 0 \to N_{R_e/\Lambda} \to N_{R_e / \PP^n} \to N_{\Lambda/\PP^n}|_{R_e} \to 0 . \]
We have $N_{\Lambda/\PP^n} = \OO_{\PP^1}(e)^{n-e}$ and by Corollary \ref{cor-RnNormalBundle}, $N_{R_e/\Lambda} = \OO_{\PP^1}(e+2)^{e-1}$. Since $$\Ext^1(N_{\Lambda/\PP^n}|_{R_e}, N_{R_e/\Lambda}) = 0,$$ we see that $N_{R_e/\PP^n}$ has the desired form.
\end{proof}

\subsection*{Families of vector bundles} Let $\mathcal{V}$ be a family of vector bundles on $\PP^1$. Suppose that the family contains a vector bundle $V_0$ with a given splitting type. Then the {\em expected codimension} of the locus of vector bundles in $\mathcal{V}$ with that  splitting type is $h^1(\End(V_0))$ (see \cite[Lemma 2.4]{Coskun} and \cite{CoskunRiedl}). 

\begin{lemma}
\label{expCodimKernels}
Let $A$ be a vector bundle on $\PP^1$ and consider the family $\mathcal{F}$ of vector bundles given by surjective kernels of homomorphisms in $\Hom(A, \OO_{\PP^1}(k))$.  If $H^1(V^* \otimes A)=0$, then the locus of bundles in $\mathcal{F}$ with the same splitting type as $V$ has expected codimension $h^1(\End(V))$. 
\end{lemma}

\begin{proof}
This lemma is standard (see \cite[Proposition 3.6]{CoskunHuizengaWBN}). Briefly, the Kodaira-Spencer map $\kappa: \Hom(A, \OO_{\PP^1}(k)) \rightarrow \Ext^1(V,V)$ naturally factors into $$\Hom (A, \OO_{\PP^1}(k)) \stackrel{\mu}{\rightarrow} \Hom(V, \OO_{\PP^1}(k)) \stackrel{\nu}{\rightarrow} \Ext^1(V,V),$$ where $\mu$ and $\nu$ are the natural morphisms that occur when applying $\Hom(-, \OO_{\PP^1}(k))$ and $\Hom(V, -)$ to the sequence $0 \rightarrow V \rightarrow A \rightarrow \OO_{\PP^1}(k) \rightarrow 0,$ respectively. Both $\mu$ and $\nu$ are surjective, since $\Ext^1(\OO_{\PP^1}(k), \OO_{\PP^1}(k))=0$ and $\Ext^1(V, A) = H^1(V^*\otimes A) =0$. Hence, the codimension of the locus of bundles with the given splitting type in the family is the same as the codimension in the versal deformation space (see \cite[Lemma 2.4]{Coskun} and \cite[\S 2]{CoskunRiedl}).
\end{proof}

\section{Normal bundles of curves on hypersurfaces of degree at least $3$}\label{sec-3}
In this section, we study the splitting type of normal bundles of rational curves on hypersurfaces of degree $d\geq 3$. The analysis of rational curves on complete intersections of hypersurfaces of degree $d=2$ requires different techniques and will be carried out in the next section. 

Let $n \geq 3$ be an integer. Let $C$ be a smooth rational curve of degree $e$ and let $X \subset \PP^n$ be a degree $d$ hypersurface containing $C$. We first obtain an explicit description of the normal bundle $N_{C/X}$. There is a standard sequence of ideal sheaves
\begin{equation}\label{eq-standardseq}
 0 \longrightarrow \II_{X/\PP^n} \stackrel{\alpha}{\longrightarrow} \II_{C/\PP^n} \longrightarrow \II_{C/X} \longrightarrow 0, 
\end{equation}
where the  map $\alpha$ is given by expressing a function vanishing on $X$ in terms of the generators of the ideal sheaf of $C$. Given a choice of degree $d$ polynomial $f$ with $X = V(f)$, we get an identification of $\II_{X/\PP^n}$ with $\OO_{\PP^n}(-d)$. Using this identification and applying $\Hom( -, \OO_{C})$ to the exact sequence (\ref{eq-standardseq}), we obtain the sequence
\begin{equation}\label{eq-2}
0 \longrightarrow N_{C/X} \longrightarrow N_{C/\PP^n} \stackrel{\psi}{\longrightarrow} \OO_{\PP^1}(ed).
\end{equation}
The map $\psi$ is the dual of the map $\alpha$ that expresses functions vanishing on $X$ in terms of the generators of the ideal of $C$. Thus, we have a map $\phi$ that sends functions $f \in H^0(\II_C(d))$ to elements of $\Hom(N_{C/\PP^n}, \OO_{\PP^1}(ed))$.

This correspondence can be identified via the sequence
\[ 0 \longrightarrow \II_{C/\PP^n}^2\otimes \OO_{\PP^n}(d) \longrightarrow \II_{C/\PP^n} \otimes \OO_{\PP^n}(d) \longrightarrow N_{C/\PP^n}^*\otimes \OO_{\PP^n}(d) \longrightarrow 0 \]
with the map 
\[ 0 \longrightarrow H^0(\II_{C/\PP^n}^2(d)) \longrightarrow H^0(\II_{C/\PP^n}(d)) \stackrel{\phi}{\longrightarrow} H^0(N_{C / \PP^n}^*(d)) .\]



The main result of this section is the following theorem.

\begin{theorem}
\label{surjectivityPhi}
Let $C \subset \PP^n$ be a rational curve of degree $e$. The map $\phi$ is surjective in the following cases:
\begin{enumerate}
\item The curve $C = R_e$ is a rational normal curve of degree $e \leq n$ and $d \geq 3$;
\item The curve $C$ is a general rational curve of degree $e= n+1$ and $d, n \geq 5$.
\end{enumerate}
\end{theorem}

Since the kernel of $\phi$ is $H^0(\II_{C/\PP^n}^2(d))$, we need to estimate the dimension of $H^0(\II_{C/\PP^n}^2(d))$. The following two lemmas provide the desired estimates.

\begin{lemma}
\label{degenLem} Let $R_e$ be a rational normal curve in $\PP^n$ of degree $e$ and let  $d \geq 3$.  Then $$h^0(\II_{R_e}^2(d)) \leq \binom{n+d}{d} - e(nd+1)+(e-1)(n+2).$$
\end{lemma}
\begin{proof}
Degenerate $R_e$ to the chain of $e$ lines $C_0 = \ell_1 \cup \dots \cup \ell_e$, where $\ell_i$ is given by the vanishing of all $x_j$ with $j \neq i-1, i$. In order for a function $f$ to vanish to order $2$ along $C_0$, it needs to vanish to order two along each line $\ell_i$. A polynomial vanishes to order $2$ along $\ell_i$ if and only if the coefficients of the monomials $$x_{i-1}^jx_i^{d-j} \ \mbox{for} \ 0 \leq j \leq d \quad \mbox{and} \quad x_m x_{i-1}^jx_i^{d-1-j} \ \mbox{for} \ 0 \leq j \leq d-1, m \neq i-1, i$$ vanish.
Vanishing to order two along $\ell_i$ imposes $d+1+(n-1)d = nd+1$ linear conditions on polynomials of degree $d$. Let $i_1 < i_2$. The conditions imposed by  $\ell_{i_1}$ and $\ell_{i_2}$ overlap if and only if  $i_2= i_1 + 1$. 
The $n+2$ terms $$x_i^d, \quad x_{i-1} x_i^{d-2} x_{i+1} \quad \mbox{and}  \quad x_j x_i^{d-1}, \ j \neq i$$ are common to both  $\ell_i$ and $\ell_{i+1}$. In total, we conclude that vanishing to order  two along $C_0$ imposes $e(nd+1)-(e-1)(n+2)$ linear conditions on a polynomial of degree $d$. The lemma follows by semicontinuity.
\end{proof}

The following lemma provides similar estimates for rational curves of degree $n+1$.

\begin{lemma}\label{lem-lastcase}
Let $d, n \geq 5$. Let $C$ be a general rational curve of degree $n+1$. Then $$h^0(\II_C^2(d)) \leq \binom{2n}{n} - (d-1)(n^2+n) -1.$$
\end{lemma}

\begin{proof}
Degenerate $C$ to a chain of $n+1$ lines $C_0 = \ell_1 \cup \dots \cup \ell_n \cup \ell_{n+1}$, where $\ell_i$ is given by the vanishing of all $x_j$ with $j \neq i-1, i$ for $1 \leq i \leq n$ and $\ell_{n+1}$ is defined by $$x_1= x_3= x_4= \cdots = x_{n-1} = x_0 - x_2 =0.$$  The proof of Lemma \ref{degenLem} determines the coefficients of the monomials that vanish for any hypersurface of degree $d$ double along $\ell_1 \cup \cdots \cup \ell_n$. We need to describe the additional conditions imposed by being double along $\ell_{n+1}$.  For simplicity, let $c_{i,j}$ denote the coefficient of the monomial $x_0^i x_2^j x_n^{d-i-j}$ and let $c_{i,j}^m$ denote the coefficient of $x_m x_0^i x_2^j x_n^{d-1-i-j}$ for $m \neq 0,2,n$.

If a hypersurface of degree $n$ is double along $\ell_{n+1}$, then the following linear relations must be satisfied among the coefficients of the hypersurface: 
\begin{equation}\label{eq-3}
\sum_{i=0}^{d-1-k} c_{i, d-1-k-i}^m =0 \ \ \mbox{for} \ \ m \not= 0,2,n, 0 \leq k \leq d-1 \ \ \mbox{and}
\end{equation}
\begin{equation}\label{eq-4}
\sum_{i=0}^{d-k} c_{i,d-k-i} =0, \quad \sum_{i=0}^{d-k-1} i c_{i, d-k-i} =0 \ \ \mbox{for} \ \ 1 \leq k \leq d-1.
\end{equation}
The relations (\ref{eq-3}) are independent  from those in Lemma \ref{degenLem} as long as $k< d-2$ or $k=d-2$ and $m \neq n-1$. To see this, note that the only monomials $x_m x_b^j x_c^k$ whose coefficients vanish by Lemma \ref{degenLem} are those with $b$ and $c$ one apart from each other, while for $k < d-2$ or $k=d-2$ and $m \neq n-1$ the $i=0$ term of the sum is different from each of these. When $m=n-1$ and $k=d-2$, the relations are already implied by the relations in Lemma \ref{degenLem}. Considering the $i=0$ and $1$ terms of each sum in the set of relations (\ref{eq-4}), we see that they are independent from each other and all of the other conditions in all characteristics. Altogether we get
$$n(nd+1)-(n-1)(n+2) + (n-2)(d-2)+n-3 + 2(d-1)= (d-1)(n^2+n)+1$$ independent conditions. We conclude that $$h^0(\II_{C_0}^2(n)) = \binom{2n}{n} - (d-1)(n^2+n)-1.$$ The Lemma follows by semicontinuity.
\end{proof}

\begin{proposition}\label{prop-surjective}
If $d \geq 3$, then  $H^0(\II_{R_e}(d)) \stackrel{\phi}{\to} \Hom(N_{R_e/\PP^n}, \OO_{\PP^1}(ed))$ is surjective for a rational normal curve $R_e$ of degree $e \leq n$ in $\PP^n$.
\end{proposition}
\begin{proof}
The proof is a dimension count.  Since $R_e$ is a rational normal curve, we see that $$h^0(\II_{R_e}(d)) = \binom{n+d}{d} - (ed+1).$$  Lemma \ref{degenLem} implies that $$h^0(\II_{R_e}^2(d)) \leq \binom{n+d}{d} - e(nd+1)+(e-1)(n+2).$$ Hence, the image of $\phi$ has dimension at least
\[ \binom{n+d}{d} - (ed+1) - \left( \binom{n+d}{d} - e(nd+1)+(e-1)(n+2) \right) = end - e(n+1+d) + n + 1 . \]
Now we compute the dimension of $$\Hom(N_{{R_e}/\PP^n},\OO_{\PP^1}(ed)) = \Hom(\OO(e+2)^{e-1}\oplus \OO(e)^{n-e}, \OO(ed)).$$ This is globally generated, since $e+2 \leq ed$ for our range of degrees. Thus, the dimension is $$(e-1)(ed-e-1)+ (n-e)(ed-e+1)=  end- e(n+1+d) +n +1.$$ Thus, we see that $\phi$ is necessarily surjective. Moreover, it follows that the inequality in Lemma \ref{degenLem} is an equality.
\end{proof}

For studying complete intersections, we will need to extend Proposition \ref{prop-surjective} inductively.

\begin{proposition}\label{passingtohypersurface}
Let $X$ be a general complete intersection in $\PP^n$ of type $(d_1, \dots, d_k)$, degree $d= \sum_{i=1}^k d_i$, and  dimension at least $3$ containing a rational normal curve $R_e$. Let $d_{k+1} \geq 3$ be an integer such that $d_i \leq d_{k+1}$ for $1 \leq i \leq k$. Then for a general hypersurface $Y$ of degree $d_{k+1}$ containing $R_e$, $N_{R_e/Y \cap X}$ corresponds to the kernel of a general element of $\Hom(N_{R_e/X}, \OO_{\PP^1}(ed_{k+1}))$. In particular, if $N_{R_e/X}$ is balanced, then $N_{R_e/Y \cap X}$ is also balanced.
\end{proposition}

\begin{proof}
Applying $\Hom(-, \OO_{\PP^1}(ed_{k+1}))$ to the standard exact sequence $$0 \rightarrow N_{R_e/X} \rightarrow N_{R_e/\PP^n} \rightarrow N_{X/\PP^n}|_{R_e} \cong \bigoplus_{i=1}^k \OO_{\PP^1}(ed_i) \rightarrow 0,$$ we obtain a long exact sequence
$$\Hom(N_{R_e/\PP^n}, \OO_{\PP^1}(ed_{k+1})) \stackrel{\psi}{\rightarrow} \Hom(N_{R_e/X}, \OO_{\PP^1}(ed_{k+1})) \rightarrow \Ext^1 \left(\bigoplus_{i=1}^k \OO_{\PP^1}(ed_i), \OO_{\PP^1}(ed_{k+1}) \right).$$ Since $d_{k+1} \geq d_i$, we conclude that the $\Ext^1$-term vanishes and the  map  $\psi$ is surjective. By Proposition \ref{prop-surjective}, the general $Y$ induces a general homomorphism in $\Hom(N_{R_e/\PP^n}, \OO_{\PP^1}(ed_{k+1}))$. By the surjectivity of $\psi$, this $Y$ also induces a general element of $\Hom(N_{R_e/X}, \OO_{\PP^1}(ed_{k+1}))$. Hence, for a general $Y$, $N_{R_e/Y \cap X}$ is given as the kernel 
$$0 \rightarrow N_{R_e/Y \cap X} \rightarrow N_{R_e/X} \rightarrow \OO_{\PP^1}(ed_{k+1}) \rightarrow 0$$ for a general homomorphism. 

If $N_{R_e/X}$ is balanced, then its summands have degrees $\lfloor \frac{e(n-d+1)-2}{n-k-1} \rfloor$ and $\lceil \frac{e(n-d+1)-2}{n-k-1} \rceil$. Since $d_{k+1} \geq 3$, we have $ed_{k+1} > \frac{e(n-d+1)-2}{n-k-1}$. Lemma \ref{lem-generalkernel} (2) implies the last statement.\end{proof}

The proof in the case $e=n+1$ is similar to the proof of Proposition \ref{prop-surjective}. 

\begin{proposition}
\label{prop-lastcase}
Let $d, n \geq 5$. Then the  map $\phi: H^0(\II_{C}(d)) \to \Hom(N_{C/\PP^n}, \OO_{\PP^1}(d(n+1)))$ is surjective for a general rational curve $C$ of degree $n+1$ in $\PP^n$.
\end{proposition}
\begin{proof}
The kernel of $\phi$ is $H^0(\II_{C/\PP^n}^2(d))$ and by Lemma \ref{lem-lastcase} has dimension at most
\[ \binom{2n}{n} - (d-1)(n^2+n)-1 .\]
Since $C$ is nondegenerate, it is $d$-regular for any $d \geq 2$ \cite{GrusonLazarsfeldPeskine} and we have
\[ h^0(\II_{C/\PP^n}(d)) = \binom{2n}{n} - (d(n+1)+1) .\]
Thus, the image of $\phi$ has dimension at least $(n+1)(nd-n-d)$. 

By the Euler sequence for $\PP^n$ and the sequence defining $N_{C/\PP^n}$, every factor of $N_{C/\PP^n}$ has degree at least $n+1$.  Since the total degree is $(n+1)^2 -2$, every summand of $N_{C/\PP^n}$ has degree at most $3n+1$.  Since $n, d \geq 5$, we see that $3n+1 \leq ed=(n+1)d$, which means $\Hom(N_{C/\PP^n}, \OO_{\PP^1}((n+1)d))$ is globally generated, and hence has dimension
\[ (n-1) + (n-1)(n+1)d - (n+1)^2+2 = (n+1)(nd-n-d) .  \]
We conclude that $\phi$ is surjective.
\end{proof}

\begin{proof}[Proof of Theorem \ref{surjectivityPhi}]
The surjectivity of $\phi$ in case (1) is proved in Proposition \ref{prop-surjective} and in case (2) in Proposition \ref{prop-lastcase}. The theorem follows.
\end{proof}

When $d=n=4$, the map $\phi$ is not surjective. For completeness, assuming $p \not= 5$, we  exhibit a rational curve of degree $5$ in a quartic threefold that has balanced normal bundle.

\begin{proposition}\label{4n+1}
Assume $p\not= 5$. Then the general rational curve of degree $5$ has a balanced normal bundle in a quartic threefold.
\end{proposition}

\begin{proof}
It suffices to exhibit one rational curve $C$ and one hypersurface $X$ containing $C$ such that $N_{C/X}$ is balanced. 
Let $C$ be the image of the map $$f= [s^5, s^4t, s^2 t^3, st^4, t^5]: \PP^1 \rightarrow \PP^4.$$
We can compute the normal bundle of $C$ from the kernel of $\partial f$. This computation depends on the characteristic. 

$$\partial f = \left( \begin{array}{ccccc} 5s^4 & 4s^3 t & 2s t^3 & t^4 & 0 \\ 0 & s^4 & 3 s^2 t^2 & 4 st^3 & 5 t^4 \end{array} \right).$$ When $p=2$, many of the entries vanish and there are the following relations among the columns $C_0, \dots, C_4$ of $\partial f$:
$$t^4 C_0 = s^4 C_3, \quad  t^2 C_1 = s^2 C_2, \quad t^2 C_2 = s^2 C_4.$$ Consequently, the normal bundle is given by $$N_{C/\PP^4} = \OO_{\PP^1}(7) \oplus \OO_{\PP^1}(7) \oplus \OO_{\PP^1}(9).$$ In characteristics $p \not= 2, 3, 5$, the 
relations are $$2t^3C_0 -3st^2 C_1 + s^3 C_2= t^3 C_1 -3s^2t C_2 + 2s^3 C_3=  t^2 C_2 -2st C_3 +t^2C_4 = 0.$$ When $p=3$, the relations change slightly but their degrees remain the same. Consequently, when $p \not= 2,5$, the
splitting is $\OO_{\PP^1}(7) \oplus \OO_{\PP^1}(8)^2.$

Now we compute $N_{C/X}$ for a general $X$ containing $C$. The calculation is most interesting when $p=2$.  A simple Macaulay2 \cite{M2} calculation shows that the ideal of $C$ is generated by 
$$F_1= z_3^2 - z_2 z_4, \quad F_2= z_1 z_3-z_0z_4, \quad F_3 = z_2^2-z_1z_4, \quad F_4= z_1z_2 - z_0z_3, \quad F_5= z_1^3-z_0^2 z_2$$
and these generators satisfy the following syzygies 
$$\left( \begin{array}{ccccc}  z_1 & -z_3 & 0 & z_4 & 0 \\ z_0 & -z_2 & 0 & z_3 & 0 \\ 0 & z_2^2-z_1z_4 & z_0z_4 - z_1z_3 & 0 & 0 \\ z_2^2 & -z_3 z_4 & z_2z_4 - z_3^2 & z_4^2 & 0 \\ 0 & 0 & z_1z_2 - z_0 z_3 & z_1z_4 - z_2^2 & 0 \\ 0 & z_0 z_1 & - z_0^2 & z_1^2  & - z_2 \\  0 & z_0 z_2 & -z_1^2 & z_1 z_2 &- z_4 \\ 0 & z_1^2 &  -z_0 z_1 & z_0 z_2 & -z_3 \end{array} \right) \left( \begin{array}{c} F_1 \\ F_2 \\ F_3 \\ F_4 \\ F_5 \end{array}\right) = 0.$$

Let $\alpha \in N_{C/\PP^4} = \Hom(\II_{C/\PP^n}, \OO_C)$. Express a hypersurface $F$ containing $C$ as $\sum a_i F_i =0$. We would like to find generators of the normal bundle. From the syzygies given by the first and sixth rows of the matrix, we have the relations
$$ s^4t \alpha(F_1) - st^4 \alpha(F_2) + t^5 \alpha (F_4) = 0$$
$$ s^9 t \alpha(F_2) - s^{10} \alpha(F_3) + s^8 t^2 \alpha(F_4) - s^2 t^3 \alpha(F_5) = 0$$

The first equation implies that we can find $g_{1,7}$ and $g_{4,9}$ so that
 $$\alpha(F_1) = t^3 g_{1, 7}, \quad \alpha(F_4) = s g_{4,9}, \quad \mbox{and} \quad \alpha(F_2) = s^3   g_{1, 7} + t g_{4,9}.$$ The second equation in characteristic 2 implies that we can find $g_{5,7}$ so that
$$\alpha(F_5) = s^8 g_{5, 7}, \quad \mbox{and} \quad \alpha(F_3) = s^2 t g_{1,7} - t^3 g_{5,7}.$$ We can take the local generators of $N_{C/\PP^n}$ to be $g_{1,7}, g _{4,9}$ and $g_{5,7}$, homogeneous polynomials of degrees $7,9$ and $7$, respectively. When the characteristic is not 2, the second relation gives shows that $s$ divides $tg_{5,8}-2g_{4,9}$. Letting $g_{4,8}$ be defined by $sg_{4,8} = -tg_{5,8}+2g_{4,9}$, we have $$\alpha(F_5) = s^7 g_{5, 8}, \quad 2 g_{4,9} = s g_{4,8} + t g_{5,8}, \quad \alpha(F_3) = s^2t g_{1,7} + t^2 g_{4,8}.$$ In this case, we can take the local generators to be $g_{1,7}, g_{4,8}, g_{5,8}$ of degrees $7, 8$ and $8$.

In characteristic 2, we can then explicitly write down the map $\Psi$ in the exact sequence
$$0 \longrightarrow N_{C/X} \longrightarrow N_{C/\PP^4} \stackrel{\Psi}{\longrightarrow}  N_{X/\PP^4}|_C.$$ For simplicity assume that the equation of the hypersurface $X$ is given by $$F= Q_{1,2} F_1 + Q_{4,2} F_4 + L_{5,1} F_5=0,$$ where $Q_{1,2}, Q_{4,2}$ are homogeneous polynomials of degree $2$ and $L_{5,1}$ is a homogeneous polynomial of degree $1$. Then $$\Psi = [ Q_{1,2}|_C t^3, Q_{4,2}|_C s, L_{5,1}|_C s^8] : \OO_{\PP^1}(7) \oplus \OO_{\PP^1}(9) \oplus \OO_{\PP^1}(7) \longrightarrow \OO_{\PP^1}(20).$$ If we take $$Q_{1,2}|_C = t^{10}, \quad Q_{4,2}|_C= s^5 t^5, \quad L_{5,1}|_C= s^5,$$ then the kernel is given by $\OO_{\PP^1}(1) \oplus \OO_{\PP^1}(2)$. If $X$ is the hypersurface defined by  $$z_4^2 F_1 + z_0 z_4 F_4 + z_0 F_5=0,$$ then we get the desired $N_{C/X}$. When the characteristic is not 2, a similar argument gives $$\Psi = [Q_{1,2}|_C t^3, 2^{-1} Q_{4,2}|_C s^2 , 2^{-1}Q_{4,2}|_C st + L_{5,1}|_C s^7] :\OO_{\PP^1}(7) \oplus \OO_{\PP^1}(8)^2  \longrightarrow \OO_{\PP^1}(20).$$ If $X$ is the hypersurface defined by $z_4^2 F_1 + 2 z_2^2 F_4 + z_0 F_5 =0$, then $N_{C/X}$ is balanced. This concludes the proof.
\end{proof}

\subsection*{Applications of Theorem \ref{surjectivityPhi}} We now give several applications of Theorem \ref{surjectivityPhi} to separable rational connectedness. These results will be later subsumed by more general theorems. We include them here to emphasize their simplicity.

\begin{corollary}\label{hyp}
Let $X \subset \PP^n$ be a general  hypersurface of degree $d \geq 3$ containing a rational normal curve $R_e$. Then $N_{R_e/X}$ is balanced. Moreover, if $X$ is Fano and $e \geq \frac{n}{n-d+1}$, then $N_{R_e/X}$ is very ample. 
\end{corollary}

\begin{proof}
By Theorem \ref{surjectivityPhi}, the map $\phi$ surjects onto $\Hom(N_{R_e/\PP^n}, \OO_{\PP^1}(ed))$. Since $N_{R_e/\PP^n}\cong \OO_{\PP^1}(e+2)^{e-1} \oplus \OO_{\PP^1}(e)^{n-e}$, by Corollary \ref{cor-generalkernelnormalbundle}, the kernel of a general homomorphism is balanced.  This kernel is $N_{R_e/X}$, hence $N_{R_e/X}$ is balanced. A vector bundle on $\PP^1$ is very ample if all the summands have positive degree. Since the bundle is balanced, this happens precisely when  its degree is at least its rank, equivalently when $e(n-d+1)-2 \geq n-2$. Since $n-d+1 > 0$, this is equivalent to $e \geq \frac{n}{n-d+1}$.
\end{proof}

\begin{corollary}
\label{cor-Zhu}
A general Fano hypersurface of degree $d$ in $\PP^n$ is separably rationally connected.  Furthermore, such a hypersurface contains a very free curve of degree $\lceil \frac{n}{n-d+1}\rceil$.
\end{corollary}

\begin{proof}
Quadric hypersurfaces contain very free conics. For hypersurfaces of degree $d \geq 3$, the corollary follows from the previous corollary.
\end{proof}

Zhu proved Corollary \ref{cor-Zhu} in the special case $n=d$. The argument in Corollary \ref{hyp} together with Proposition \ref{prop-lastcase} shows the following corollary.

\begin{corollary} \label{cor-nPlusOneCase}
Let $X \subset \PP^n$ be a general degree $d$ hypersurface containing a general degree $n+1$ rational curve $C$. Assume that $d, n \geq 5$. Then $N_{C/X}$ is balanced.
\end{corollary}

More generally, the same considerations apply to complete intersections. 

\begin{corollary}\label{no2s}
Let $X$ be a general complete intersection containing $R_e$. Assume $X$ has dimension at least $2$, multidegree $(d_1, \dots, d_k)$ with $d_i \geq 3$ for all $i$ and degree $d= \sum_{i=1}^k d_i$. Then $N_{R_e/X}$ is balanced. Moreover, if $X$ is Fano and  $e \geq \frac{n+1-k}{n-d+1}$, then $N_{R_e/X}$ is very ample.
\end{corollary}

\begin{proof}
 We  prove $N_{R_e/X}$ is balanced by induction on $k$. Corollary \ref{hyp} shows the base case $k=1$. Assume that the corollary holds for complete intersections of $k-1$ hypersurfaces. List the degrees in increasing order $d_1 \leq \cdots \leq d_k$. Let $X'$ be a general complete intersection of type $(d_1, \dots, d_{k-1})$. By the induction hypothesis, $N_{R_e/X'}$ is balanced.  Since $d_k \geq d_i$ for $1 \leq i \leq k-1$ and $d_k \geq 3$, Proposition \ref{passingtohypersurface} shows that $N_{R_e/X}$ is also balanced. Since the rank of $N_{R_e/X}$ is $n-k-1$ and its degree is $e(n-d+1)-2$, the last statement follows.
\end{proof}

When $d_i \geq 3$, we recover the theorem of Z. Tian \cite{Tian},  Q. Chen and Y. Zhu \cite{ChenZhu} with sharp degree bounds.

\begin{corollary}\label{cino2s}
A general Fano complete intersection of multidegree $(d_1, \dots, d_k)$ with $d_i \geq 3$ for all $i$ is separably rationally connected. Furthermore, such a complete intersection contains a very free curve of degree $\lceil \frac{n+1-k}{n-d+1}\rceil$.
\end{corollary}

Unfortunately, when $d=2$ the map $\phi: H^0(\II_{R_e}(2)) \rightarrow H^0(N_{R_e/\PP^n}^*(2))$ is not surjective. In order to extend these results to complete intersections where some of the degrees are 2, we will have to study normal bundles of rational curves in complete intersections of quadrics more carefully. We will do so in the next section. 

Theorem \ref{surjectivityPhi} also has applications to the universal incidence correspondence. Let  $$S = \{(R_e,X) | \: R_e \subset X,  X \text{ is a degree $d$ hypersurface} \} $$ denote the universal incidence correspondence.

\begin{theorem}
Given a vector bundle $E$ on $\PP^1$, let $S_E \subset S$ be the space of pairs $(R_e,X)$ so that $N_{R_e/X} = E$. If $h^1(E^* \otimes N_{R_e/X})=0$, then the codimension of $S_E$ in $S$ is the expected codimension $h^1(\End(E))$.
\end{theorem}
\begin{proof}
This is proven by analyzing the fibers of the map $$\pi:S \to \{R_e \subset \PP^n \text{ a degree $e$ rational normal curve} \} .$$
The fiber of $\pi$ is simply the set of degree $d$ hypersurfaces containing $R_e$. By Theorem \ref{surjectivityPhi}, we see that the codimension of $S_E$ is the same as the codimension of the space of maps $N_{R_e/\PP^n} \to \OO_{R_e}(d)$ having kernel $E$. The fact that this last codimension is $h^1(\End(E))$ follows from Lemma \ref{expCodimKernels}.
\end{proof}

\section{Normal bundles of rational curves in complete intersections of quadric hypersurfaces}\label{sec-quadrics}
In this section, we study normal bundles of rational curves in complete intersections of quadric hypersurfaces. Some calculations in this section are easier to carry out using Macaulay2. We have included sample calculations at \url{http://homepages.math.uic.edu/~coskun/M2nb.txt}.

Recall that $R_n$, the rational normal curve of degree $n$ in $\PP^n$, is cut out by the quadrics $q_{i,j} = x_i x_{j-1} - x_{i-1}x_j$ for $1 \leq i< j \leq n$.  We can use Proposition \ref{RNCnormalBundle} to compute normal bundles of rational normal curves in a complete intersection of quadric hypersurfaces.  Let $X= V(f)$ be a hypersurface of degree $d \geq 2$ containing $R_n$. Then $X$ induces a map $\psi_f: N_{R_n/\PP^n} \to \OO_{R_n}(d)$. 
Proposition \ref{RNCnormalBundle} allows us to compute the map $\psi_f$ explicitly. Write $f = \sum_{i,j} a_{i,j} q_{i,j}$. Then $\psi_f(\alpha) = \sum_{i,j} a_{i,j}|_{R_n} \alpha(q_{i,j})$. By Proposition \ref{RNCnormalBundle}, we can express the latter sum as $$\psi_f(\alpha) = \sum_{i,j} a_{i,j}|_{R_n} \sum_{\ell =i}^{j-1} s^{n-j-i+\ell}t^{j+i-\ell-2} b_{\ell, \ell+1}.$$ Collecting the coefficients of $b_{i, i+1}$, we can rewrite this expression as $\sum_{i=1}^{n-1} c_i b_{i, i+1}.$ The map $\psi_f$ in this basis is given by the matrix $(c_1, \dots, c_{n-1}).$

\begin{corollary}\label{cor-balanced2}
A general quadric hypersurface in $\PP^n$ contains a degree $n$ rational curve with balanced normal bundle $\OO_{\PP^1}(n+1)^{n-2}$.
\end{corollary}
\begin{proof}
It suffices to exhibit one quadric hypersurface containing the rational normal curve that has balanced normal bundle. Take $f = \sum_{i=1}^{n-1} q_{i,i+1}$, so that the map $$N_{R_n/\PP^n}\cong \OO_{\PP^1}(n+2)^{n-1} \longrightarrow N_{X/\PP^n}|_{R_n} \cong \OO_{\PP^1}(2n)$$ has coordinates $(s^{n-2}, s^{n-3}t, \dots, t^{n-2})$. The columns $c_i$ of this matrix satisfy the $n-2$ linear relations $$t c_i - s c_{i+1} = 0, \ \ \mbox{for} \ \ 1 \leq i \leq n-2.$$ Hence, the kernel is $\OO_{\PP^1}(n+1)^{n-2}$ as desired.
\end{proof}

Similarly, we can compute the normal bundle of a rational normal curve in the complete intersection of two or three quadric hypersurfaces explicitly to see that the general one has balanced normal bundle.

\begin{proposition}\label{prop-balanced22}
Let $n \geq 5$ and $X$ be a general complete intersection of type $(2,2)$ containing $R_n$. Then $N_{R_n/X}$ is balanced.
\end{proposition}
\begin{proof}
To simplify notation, set $k_1=\lfloor \frac{n-2}{2} \rfloor$ and $k_2 = \lceil \frac{n-2}{2} \rceil$. Let 
\[ f_1 = q_{1,2} + q_{3,4} + \dots + q_{2k_1-1,2k_1} + q_{n-1,n}, \ \ \mbox{and} \]  \[ f_2 = q_{1,3} + q_{4,5}+q_{6,7} + \dots + q_{2k_2-2,2k_2-1} + q_{n-2,n}. \]
Let $X= V(f_1, f_2)$. We claim that $N_{R_n/X}$ is balanced. The exact relations depend on the parity of $n$. We explicitly write out the argument when $n$ is even and leave the easy modifications for odd $n$ to the reader.

If $n$ is even, then the map $$N_{R_n/\PP^N} \cong \OO_{\PP^1}(n+2)^{n-1} \to N_{X/\PP^n}|_{R_n} \cong \OO_{\PP^1}(2n)^2$$ has coordinates
\[ \left( \begin{array}{llllllllll}
s^{n-2} & 0 & s^{n-4}t^2 & 0 & \dots & 0 & s^3t^{n-5} & 0 & 0 & t^{n-2} \\
s^{n-3}t & s^{n-2} & 0 & s^{n-5}t^3 & \dots & s^4t^{n-6} & 0 & s^2t^{n-4} & t^{n-2} & st^{n-3}
\end{array} \right) \]
We let $c_i$ for $i = 1, \dots, n-1$ be the $i$th column of the map. The columns of the matrix satisfy the degree three relations
\[ -st^2c_1 + t^3c_2 + s^3c_3 = 0  \]
\[ t^3 c_{n-4} + s^2t c_{n-3} - s^3 c_{n-1} = 0\]
and the degree two relations
\[ t^2 c_1 - s^2 c_3 - s^2 c_4 = 0,  \]
\[ t^2c_{n-3} - s^2c_{n-2} = 0, \ \ \mbox{and} \]
\[ t^2c_i - s^2 c_{i+2} = 0, \ \ \mbox{for} \ \ i = 3, \dots, n-5. \ \ \]
The matrix of relations has two $(n-3) \times (n-3)$ minors equal to $t^{2n-4}$ and $-s^{2n-4}$, which do not simultaneously vanish. We conclude that  these relations generate the set of relations and the kernel is isomorphic to $$N_{R_n/X} \cong\OO_{\PP^1}(n-1)^2 \oplus \OO_{\PP^1}(n)^{n-5}.$$
\end{proof}

\begin{proposition}\label{prop-balanced222}
Let $n \geq 6$ and $X$ be a general complete intersection of type $(2,2,2)$ containing $R_n$. Then $N_{R_n/X}$  is balanced.
\end{proposition}

\begin{proof}
The cases $6 \leq n \leq 10$ can be checked explicitly by a simple Macaulay2 \cite{M2} calculation. We will assume that $n \geq 11$. The calculation is similar to the calculation in Proposition \ref{prop-balanced22} and depends on $n$ mod 3. Set $u = \lfloor \frac{n}{3} \rfloor$. If $n$ is $0$ mod 3, then let 
$$f_1 = q_{1,2} + q_{5,6} + \sum_{i=0}^{u -4} q_{7+3i, 8+3i}  + q_{n-1, n},$$
$$f_2 = q_{1,3} + q_{4,6} + \sum_{i=0}^{u -4} q_{8+3i, 9+3i} + q_{n-2, n},$$
$$f_3 = q_{1,4} +  q_{5,7} + \sum_{i=0}^{u -5} q_{9+3i, 10+3i} + q_{n-3, n},$$

If $n$ is 1 mod 3, then let 
$$f_1 = q_{1,2} + q_{5,6} + \sum_{i=0}^{u -5} q_{7+3i, 8+3i} + q_{n-5, n-4} + q_{n-1, n},$$
$$f_2 = q_{1,3} + q_{4,6} + \sum_{i=0}^{u -5} q_{8+3i, 9+3i} + q_{n-6, n-4} + q_{n-2, n},$$
$$f_3 = q_{1,4} +  q_{5,7} + \sum_{i=0}^{u -4} q_{9+3i, 10+3i} + q_{n-3, n}.$$

Finally if $n$ is $2$ mod 3, then let
$$f_1 = q_{1,2} + q_{5,6} + \sum_{i=0}^{u -3} q_{7+3i, 8+3i} + q_{n-1, n},$$
$$f_2 = q_{1,3} + q_{4,6} + \sum_{i=0}^{u -4} q_{8+3i, 9+3i} + q_{n-2, n},$$
$$f_3 = q_{1,4} +  q_{5,7} + \sum_{i=0}^{u -4} q_{9+3i, 10+3i} + q_{n-3, n}.$$

The calculations are almost identical in the three cases. We will sketch the details in the case $n$ is divisible by $3$ and leave minor modifications to the reader.
Let $X=V(f_1, f_2, f_3)$. We would like to compute $N_{R_n/X}$ as the kernel of the map $N_{R_n/\PP^n} \cong \OO_{\PP^1}(n+2)^{n-1} \rightarrow N_{X/\PP^n}|_{R_n} \cong \OO_{\PP^1}(2n)^3$ given by the matrix 
$$\left( \begin{array}{ccccccccccc}  s^{n-2} & 0 & 0 & 0 & s^{n-6}t^4 & 0 & s^{n-8}t^6 & 0 & 0 & s^{n-11} t^9 & \cdots \\
 s^{n-3}t & s^{n-2} & 0 & s^{n-6}t^4 & s^{n-5}t^3 & 0 & 0 & s^{n-9} t^7 & 0 & 0 & \cdots  \\
s^{n-4} t^2 &  s^{n-3} t & s^{n-2} & 0 & s^{n-5}t^5 & s^{n-6}t^4 & 0 & 0 & s^{n-10} t^8 & 0 & \cdots 
  \end{array} \right. $$
  
  $$\left. \begin{array}{cccccc}   \cdots &   s^4 t^{n-6} & 0 & 0 & 0 & t^{n-2} \\
  \cdots & 0 & s^3 t^{n-5} & 0 & t^{n-2} & s t^{n-3} \\
 \cdots & 0 & 0 & t^{n-2} & s t^{n-3}& s^2 t^{n-4}
  \end{array} \right).$$

The relations among the columns of this matrix can be easily determined by a Macaulay2 calculation and induction on $n$. There are  $3$ degree $4$ relations between the first 7 columns and 3 degree 4 relations between the last 6 columns and another $n -10$ degree 3 relations among the columns. The coefficients of all the polynomials occurring in these relations are $-1, 0$ or $1$. If we replace $n$ by $n+3$, the first 7 and the last 6 columns of the new matrix satisfy the same degree 4 relations and we introduce three new binomial relations of degree $3$. We conclude that the kernel is given by $\OO_{\PP^1}(n-2)^6 \oplus \OO_{\PP^1}(n-1)^{n-10}$. This concludes the proof. 
\end{proof}

A similar calculation yields the following proposition. 

\begin{proposition}\label{prop-2k}
The normal bundle of $R_{2k+1} \subset \PP^{2k+1}$ in a general complete intersection $X$ of $k$ quadrics containing $R_{2k+1}$ is balanced.
\end{proposition}

\begin{proof}
Let $f_i = q_{1,i} + q_{2k-i+1, 2k+1}$ for $1 \leq i \leq k$. Let $X$ be the intersection of the $k$ quadrics $f_1, \dots, f_k$. Then we claim that $N_{C/ X} = \OO_{\PP^1}(4)^k$. 
The normal bundle is  the kernel of the map $$N_{R_{2k+1}/ \PP^{2k+1}} \cong \OO_{\PP^1}(2k+3)^{2k} \rightarrow N_{X/\PP^{2k+1}}|_{R_{2k+1}} \cong \OO_{\PP^1}(4k+2)^k$$ given by the matrix
$$\left( \begin{array}{cccccccccc} s^{2k-1}  & 0 & 0 & 0 & 0 & 0 & 0 & 0 & 0 & t^{2k-1} \\   s^{2k-2} t  & s^{2k-1} & 0 & 0 & 0 & 0 & 0 & 0 & t^{2k-1} & s t^{2k-2} \\   &   &  &\cdots & \cdots&\cdots &\cdots & & &  \\     s^k t^{k-1} & \cdots &  s^{2k-3} t^2 & s^{2k-2} t & s^{2k-1} & t^{2k-1} & s t^{2k-2} & s^2 t^{2k-3} & \cdots & s^{k-1}t^k\end{array} \right).$$
The columns $c_i$ of this matrix satisfy the relations $$ \sum_{j=0}^{i-1}  s^{i-j-1} t^{2k-i+j} c_{k-j}  - \sum_{j=1}^i   s^{2k-i+j -1} t^{i-j} c_{k+j}  \ \ \mbox{for} \ \ 1 \leq i \leq k.$$ This matrix of relations has two minors equal to $t^{k(2k-1)}$ and $(-1)^k s^{k(2k-1)}$, which do not simultaneously vanish on $R_{2k+1}$. We conclude that the kernel of the map is given by  $\OO_{\PP^1}(4)^k$ as desired.
\end{proof}

As a final illustration of these techniques, we sketch the proof of  the following proposition.

\begin{proposition}
\label{degree2kProp}
Let $X$ be a general complete intersection of $k$ quadrics in $\PP^{2k}$ containing $R_{2k}$. Then the normal bundle of $R_{2k}$ on $X$ is balanced.
\end{proposition}
\begin{proof}
Let $f_i = q_{1, i+1} + q_{2k-i, 2k}$ for $1 \leq i \leq k-1$. Let $f_k = q_{1, 2k}$. Let $X= V(f_1, \dots, f_k)$. Then we claim that $N_{R_{2k}/X} \cong \OO_{\PP^1}(2)^{k-1}$. The normal bundle is the kernel of the map $$N_{R_{2k}/ \PP^{2k}} \cong \OO_{\PP^1}(2k+2)^{2k-1} \rightarrow N_{X/\PP^{2k+1}}|_{R_{2k}} \cong \OO_{\PP^1}(4k)^k$$ given by the matrix
$$\left( \begin{array}{ccccccccc} s^{2k-2}  & 0 & 0 & 0 &  \cdots  & 0 & 0 & 0 & t^{2k-2} \\   s^{2k-3} t  & s^{2k-2} & 0 & 0  & \cdots  & 0 & 0 & t^{2k-2} & s t^{2k-3} \\  
&    &\cdots & \cdots&\cdots &\cdots & \cdots  & &  \\  s^{k} t^{k-2} & s^{k+1}t^{k-3} & \cdots & s^{2k-2} & 0 & t^{2k-2} & \cdots & s^{k-3} t^{k+1} & s^{k-2} t^k  \\   
t^{2k-2} & s t^{2k-3} & \cdots & s^{k-2} t^k & s^{k-1} t^{k-1} & s^k t^{k-2} & \cdots  & s^{2k-3} t & s^{2k-2} \end{array} \right).$$
It is easy to check that the columns of this matrix satisfy the relations $$ st^{2k-1} c_{k-1} +  (s^{2k} - t^{2k}) c_k - s^{2k-1} t c_{k+1}$$ and $$st^{2k-1}c_{k-1-i}  - t^{2k}c_{k-i}  + s^{2k}c_{k+i}  - s^{2k-1} t c_{k+i+1}, \ \ \mbox{for} \ \ 1 \leq i \leq k-2.$$ The proposition follows. 
\end{proof}

Given Corollary \ref{cor-balanced2} and Propositions \ref{prop-balanced22}, \ref{prop-balanced222},  \ref{prop-2k} and \ref{degree2kProp},  it is natural to conjecture the following.

\begin{conjecture}\label{conj-balanced}
Let $X$ be a general Fano complete intersection of $k$ quadric hypersurfaces in $\PP^n$ containing the rational normal curve $R_n$. Then $N_{R_n/X}$ is balanced.
\end{conjecture}

We have checked this conjecture for $k=4$, $n \leq 19$ and $k=5$, $n \leq 16$ by computer. However, we were not able to prove Conjecture \ref{conj-balanced} by direct calculation. In Proposition \ref{gluequadrics} we will prove a slightly weaker statement. For the rest of the section, we study the base cases in detail.

\begin{proposition}
\label{hyperplaneProp}
Let $R_e \subset X_0 \subset \PP^{n-1} \subset \PP^n$, where $R_e$ is a rational normal curve of degree $e$ and $X_0$ is a complete intersection. Then any extension 
\[ 0 \to N_{R_e/X_0} \to N_{R_e/X} \to N_{X_0/X}|_{R_e} = \OO_{\PP^1}(e) \to 0 \]
is realized by some complete intersection $X \subset \PP^n$ of the same multidegree as $X_0$. In particular, if $N_{R_e/X_0}$ is balanced and for $X$ a general complete intersection in $\PP^n$ containing $X_0$ and of the same multidegree as $X_0$, $N_{R_e/X}$ is balanced as well.
\end{proposition}
\begin{proof}
We have the following commutative diagram.

\catcode`\@=8
\newdimen\cdsep
\cdsep=2em

\def\cdstrut{\vrule height .2\cdsep width 0pt depth .1\cdsep}
\def\@cdstrut{{\advance\cdsep by 2em\cdstrut}}

\def\arrow#1#2{
  \ifx d#1
    \llap{$\scriptstyle#2$}\left\downarrow\cdstrut\right.\@cdstrut\fi
  \ifx u#1
    \llap{$\scriptstyle#2$}\left\uparrow\cdstrut\right.\@cdstrut\fi
  \ifx r#1
    \mathop{\hbox to \cdsep{\rightarrowfill}}\limits^{#2}\fi
  \ifx l#1
    \mathop{\hbox to \cdsep{\leftarrowfill}}\limits^{#2}\fi
}
\catcode`\@=10

\cdsep=2em
\begin{equation}\label{D1}
\begin{matrix}
& & 0 & & 0\cr
& & \arrow{d}{} & & \arrow{d}{} \cr
0 & \arrow{r}{} & N_{R_e/X_0}  & \arrow{r}{i} & N_{R_e/\PP^{n-1}} & \arrow{r}{a} & N_{X_0/\PP^{n-1}}|_{R_e} & \arrow{r}{} & 0          \cr
& &  \arrow{d}{} & & \arrow{d}{} & & \arrow{d}{=} \cr
0 & \arrow{r}{} & N_{R_e/X} & \arrow{r}{} & N_{R_e/\PP^{n-1}} \oplus \OO_{\PP^1}(e) & \arrow{r}{(a,b)} & N_{X/\PP^n}|_{R_e} & \arrow{r}{} & 0 \cr
& & \arrow{d}{} & & \arrow{d}{} \cr
& & \OO_{\PP^1}(e) & \arrow{r}{=} & \OO_{\PP^1}(e) \cr
& & \arrow{d}{} & & \arrow{d}{} \cr
& & 0 & & 0 \cr
\end{matrix}
\end{equation}

For any homogeneous polynomial $f(s,t)$ of degree $es$ in two variables,  there exists a homogeneous polynomial $F$ of degree $s$ in $n+1$ variables whose restriction to $R_e$ is $f(s,t)$. Since the roots of $f(s,t)$ on $R_e$ are in general linear position, in fact, we may take $F$ to be a product of $s$ linear forms.   Given $R_e$ and $X_0$, the first row of Diagram \ref{D1} is completely determined. Given the map $b$, we can find a complete intersection $X$ giving rise to Diagram \ref{D1}. Suppose $X_0$ is defined by equations $F_1= \cdots= F_k=0$  and the hyperplane $\PP^{n-1}$ is defined by $z_n =0$. Let $G_1, \dots, G_k$ be polynomials of degrees $d_1 -1, \dots, d_k -1$ whose restrictions to $R_e$  define the maps $\OO_{\PP^1}(e) \stackrel{b}{\rightarrow} \oplus_{i=1}^k \OO_{\PP^1}(ed_i)$. Then we can take $X$ to be the complete intersection defined by $F_1 + z_n G_1 = \cdots = F_k + z_n G_k =0$.

It remains to show that the map $b \mapsto N_{R_e/X}$ which maps $$\Hom(\OO_{\PP^1}(e), N_{X/\PP^n}|_{R_e}) \cong \Hom(\OO_{\PP^1}(e), N_{X/\PP^{n-1}}|_{R_e}) \to \Ext^1(\OO_{\PP^1}(e), N_{R_e/X_0})$$ is surjective. This follows since taking $\Hom(\OO_{\PP^1}(e), -)$ of the second  horizontal sequence gives
\[ \Hom(\OO_{\PP^1}(e), N_{X/\PP^n}|_{R_e}) \to \Ext^1(\OO_{\PP^1}(e), N_{R_e/X}) \to \Ext^1(\OO_{\PP^1}(e), N_{R_e/\PP^{n-1}} \oplus \OO_{\PP^1}(e)) = 0, \]
where the last term is $0$ since every summand in $N_{R_e/\PP^{n-1}}$ has degree at least $e$.
\end{proof}

\begin{proposition}\label{prop-k+2}
Let $e \leq k+2$ and let $X$ be a general complete intersection of $k$ quadrics in $\PP^{k+2}$ containing $R_e$. Then $X$ is smooth along $R_e$ and $N_{R_e/X}$ is a line bundle of the expected degree.
\end{proposition}
\begin{proof}
Since it is the complete intersection of $k$ hypersurfaces in $\PP^{k+2}$, $X$ is a surface. To prove the proposition it suffices to exhibit an $X$ containing $R_e$ with normal bundle of the expected degree. Our strategy will be to exhibit quadrics for which the induced map $N_{R_e/\PP^n} \to \OO_{\PP^1}(2e)^k$ has full rank everywhere.

We first study the case $e=k+2$. Let $f_i = q_{1,i+1} + q_{i+2,k+2}$ for $1 \leq i < k$, let $f_k = q_{1,n}$, and let  $X = V(f_1, \dots, f_k)$. The map $N_{R_{k+2}/\PP^{k+2}} \to \OO_{\PP^1}(2(k+2))^k$ is given by the matrix
\[ \left( \begin{array}{lllllll}
 s^{k}  & 0 & t^k & st^{k-1} & s^2t^{k-2} & \dots & s^{k-2}t^{2}  \\
 s^{k-1}t & s^{k} & 0 & t^k & st^{k-1} & \dots & s^{k-3}t^3 \\
 \dots & 	   & 	 & \dots & & &  \\
 s^2 t^{k-2} & s^3t^{k-3} & s^4t^{k-4} & \dots & s^k & 0 & t^k \\
 t^{k} & st^{k-1} & s^2 t^{k-2} & \dots & s^{k-2}t^2 & s^{k-1}t  & s^k \\
 \end{array}
 \right)
 \]
We show that the relation among the columns of this matrix has the expected degree, $k^2$. The generic rank of the map is $k$, since the rank is clearly $k$ for $t=0$.  Let $C_i$ for $1 \leq i \leq k+1$ be the $i$th column of the matrix. Let 
 $$a_i = s^{(k+1)(i-1) + 1}t^{k^2 - (k+1)(i-1)-1}\ \ \mbox{for} \ \ 1 \leq i \leq k\ \  \mbox{and} \ \ a_{k+1} = 0.$$  
Let $ b_i = s^{(k+1)(i-2)}t^{k^2 - (k+1)(i-2)}$ for $2 \leq i \leq k+1$ and  $b_1 = 0.$ Then we claim that
\begin{equation}\label{relation1}
\sum_{i=1}^{k+1} (a_i - b_i)  C_i =0.
\end{equation} 
Let $C_{i,j}$ be the $j$th entry of the column $C_i$. Notice that $C_{i,j} = \frac{t}{s} C_{i+1,j}$ for every pair $i,j$ with $1 \leq j \leq k$ except for $j < k$ and $i = j,j+1$. Since $a_i = \frac{s}{t}b_{i+1}$ for $1 \leq i \leq k$, it follows that 
\[ \sum_{i=1}^{k+1} a_i C_{i,k} - \sum_{i=1}^{k+1} b_i C_{i,k} = \sum_{i=1}^k a_i C_{i,k} - \sum_{i=1}^k \frac{t}{s} a_i \frac{s}{t} C_i = 0. \] 
For $j < k$, all the terms of $\sum_{i=1}^{k+1} a_i C_{i,j} - \sum_{i=1}^{k+1} b_i C_{i,j}$ will cancel out except for $a_j C_{j,j}$ and $-b_{j+2} C_{j+2,j}$, so
\[ \sum_{i=1}^{k+1} a_i C_{i,j} - \sum_{i=1}^{k+1} b_i C_{i,j}  = s^{(k+1)j}t^{k^2-(k+1)(j-1)-1}-s^{(k+1)j}t^{k^2-(k+1)(j-1)-1} = 0 .\]
Therefore, the relation (\ref{relation1}) holds. We now show that it is a minimal relation. Since $a_1 = st^{k^2-1}$, we know that any common factor of all the $a_i - b_i$ is a monomial in $s$ and $t$. However,  $s$ does not divide $a_2-b_2 = s^{k+2}t^{k^2-k-2} - t^{k^2}$ and $t$ does not divide $a_k-b_k = s^{k^2} - s^{k^2-k-2} t^{k+2}$. Thus, the relation is minimal. This completes the proof that for $e = k+2$, the normal bundle is the expected line bundle $\OO_{\PP^1}(-k^2+k+4)$.

We now do induction on $u = k+2 - e$. When $u=1$, let $f_i = q_{1, i+1} + q_{i+2, k+1}$ for $1 \leq i \leq k-1$ and let $f_{k} = q_{1,k}+x_{k+1} x_{k+2} $. The map $N_{R_e/\PP^{k+2}} \cong \OO_{\PP^1}(e+2)^{k} \oplus \OO(e) \rightarrow N_{X/\PP^{k+2}|_{R_e}} \cong \OO_{\PP^1}(2e)^k$ is given by the matrix 
\[ \left( \begin{array}{llllllll}
 s^{e-2}  & 0 & t^{e-2} & st^{e-3} & s^2t^{e-4} & \dots & s^{e-4}t^{2} & 0  \\
 s^{e-3}t & s^{e-2} & 0 & t^{e-2} & st^{e-3} & \dots & s^{e-5}t^3 & 0 \\
 \dots & 	   & 	 & \dots & & & & 0 \\
 s^2t^{e-4} & s^3 t^{e-5} & s^4 t^{e-6} & \dots & s^{e-2} & 0 & t^{e-2} & 0\\
 t^{e-2} & st^{e-3} & s^2 t^{e-4} & s^3t^{e-5} & \dots & s^{e-3}t & s^{e-2} & 0 \\
st^{e-3} & s^2 t^{e-4} & s^3 t^{e-5} & \dots & s^{e-3}t & s^{e-2} & 0 & t^{e} \\
 \end{array}
 \right),
 \]
which we show is nondegenerate. To do this, we work at an arbitrary point $[s,t]$ of $\PP^1$. If $t=0$, we see from the placement of the $s^{e-2}$'s that the matrix has full rank, which shows that $X$ is smooth along $C$ at that point. If $t \neq 0$, then we need to select $k$ columns with nonvanishing determinant. By the $u=0$ case, we can choose $k-1$ of the first $k$ columns so that the minor $M$ given by those columns and the first $k-1$ rows is nonzero. To get a nonvanishing $k \times k$ minor take these  $k-1$ columns together with the last column. Then the determinant of this square matrix is $t^e$ times $M$. By hypothesis, both $t^e$ and $M$ are nonzero, which completes our proof.

For $u > 1$, $f_i = q_{1,i+1} + q_{i+2,k+2-u}$ for $1 \leq i \leq k-u$, $f_{k-u+1} = q_{1,k+1-u}+ x_e x_{e+1}$, and $f_i = x_0 x_{i-1} + x_e x_{i}$ for $k-u+2 \leq i \leq k$. The map $N_{R_e/\PP^{k+2}} = \OO_{\PP^1}(e+2)^{e-1} \oplus \OO_{\PP^1}(e)^{k+2-e} \to N_{X/\PP^{k+2}}|_{R_e} = \OO_{\PP^1}(2e)^k$  is given by the matrix
\[ \left( \begin{array}{lllllllllll}
 s^{e-2}  & 0 & t^{e-2} & st^{e-3} & s^2t^{e-4} & \dots & s^{e-4}t^{2} & 0 & 0 & \dots & 0 \\
 s^{e-3}t & s^{e-2} & 0 & t^{e-2} & st^{e-3} & \dots & s^{e-5}t^3 & 0 & 0 & \dots & 0 \\
 \dots & 	   & 	 & \dots & & & & 0 \\
 s^2t^{e-4} & s^3 t^{e-5} & s^4 t^{e-6} & \dots & s^{e-2} & 0 & t^{e-2} & 0 & 0 & \dots & 0\\
 t^{e-2} & st^{e-3} & s^2 t^{e-4} & s^3t^{e-5} & \dots & s^{e-3}t & s^{e-2} & 0 & 0 & \dots & 0 \\
st^{e-3} & s^2 t^{e-4} & s^3 t^{e-5} & \dots & s^{e-3}t & s^{e-2} & 0 & t^{e} & 0 & \dots & 0 \\
  \dots & 	   & 	 & \dots & & & & \dots \\
 0 & 0 & 0 & \dots & 0 & 0 & \dots & 0 & s^e & t^{e} & 0 \\
 0 & 0 & 0 & \dots & 0 & 0 & 0 & \dots & 0 & s^e & t^{e} \\
 \end{array}
 \right)
 \]
 
We show that the matrix has full rank at every point $[s,t]$ of $\PP^1$ as before. If $t=0$, the pattern of $s^e$ terms implies that the matrix has full rank. If $t \neq 0$, then by induction we can find a $k-1$ by $k-1$ minor $M$ of the first $k$ columns that does not vanish. Then taking the columns from the minor $M$ together with the last column, we obtain a $k$ by $k$ minor that does not vanish.
\end{proof}

\begin{corollary}
\label{rPlus2Cor}
A general complete intersection of $k$ quadrics in $\PP^n$ with $k \leq n-2$ has rational curves of every degree $e \leq k+2$ with balanced normal bundle.
\end{corollary}
\begin{proof}
This follows immediately from Propositions \ref{hyperplaneProp} and \ref{prop-k+2}.
\end{proof}

\section{Rational curves with balanced normal bundle on Fano complete intersections}\label{sec-final}

\begin{definition}
Let  $E = \bigoplus_{i=1}^r \OO(a_i)$ be a vector bundle on $\PP^1$ with $$a_1 \leq a_2 \leq \cdots a_{j-1} < a_{j} = \cdots = a_r.$$  Let the \emph{imbalance} $i(E)$ of $E$ be $a_r -a_1$. Set $\delta_i = a_j -1 - a_i$ for $1 \leq i < j$. Let the \emph{indentation} of $E$ be $\delta(E) = \sum_{i=1}^{j-1} \delta_i$.  
\end{definition}
The main theorem of this section is the following.

\begin{theorem}\label{Thm-mainbalanced}
Let $X$ be a general Fano complete intersection in $\PP^n$ of degree $(d_1, \dots, d_k)$.
\begin{enumerate}
 \item If  at least one of the $d_i$ is not $2$, then $X$ contains rational curves of every degree $e \leq n$ with balanced normal bundle.
 \item If $d_1 = \cdots = d_k=2$ and $n \geq 2k+1$, then $X$ contains rational curves of every degree $e \leq n-1$ with balanced normal bundle.
 \item If $d_1 = \cdots = d_k=2$ and $n = 2k$, then $X$ contains rational curves of every degree $e \leq k+2$ with balanced normal bundle.
\end{enumerate}
\end{theorem}

Part (3) has already been proven in Corollary \ref{rPlus2Cor}. We prove parts (1) and (2) by deforming unions of rational curves with balanced normal bundle. This technique has been used in the literature, see \cite{Ran,GHS,HT}. We summarize the results for the reader's convenience. 
\begin{lemma}
\label{restrictedNormalBundle}
Let $C = C_1 \cup C_2$ be a union of two nonsingular rational curves meeting at a node $p$. Then:
\begin{enumerate}
\item The sections of $N_{C/X}|_{C_1}$ are rational sections of $N_{C_1/X}$ with at worst a simple pole at $p$ in the direction of $C_2$ and no other poles.
\item If $N_{C/X}|_{C_i}$ is globally generated for $i = 1,2$, then there is a smoothing of $C_1 \cup C_2$ on $X$. Furthermore, if $N_{C/X}|_{C_1} = \bigoplus_i \OO(a_i)$ and $N_{C/X}|_{C_2} = \bigoplus_i \OO(b_i)$, then $C$ smooths to a curve with normal bundle in the versal deformation space of $\bigoplus_i \OO(a_{\sigma(i)} + b_i)$ for some permutation $\sigma$ of the $a_i$.
\end{enumerate}
\end{lemma}

We first analyze the case $n=2k$ as a base for induction.

\begin{proposition}\label{glue2k}
Let $X$ be a general complete intersection of $k$ quadrics in $\PP^{2k}$. Then there exist rational curves $C$ of every degree $1 \leq e \leq 2k$ such that the imbalance $i(N_{C/X})$ is at most $2$. Furthermore, if the imbalance is 2, then we may assume that there is a unique summand of $N_{C/X}$ of highest degree.
\end{proposition}

\begin{proof}
By Corollary \ref{rPlus2Cor}, $X$ contains rational curves of degrees $e \leq k+2$ with balanced normal bundle. Therefore, to prove the proposition, we may assume $k+3 \leq e  \leq 2k$. Let $e_0 = e - (k+1)$, which by hypothesis satisfies $2 \leq e_0 \leq k-1$. Let $C_1$ be a curve of degree $e_0$ with balanced normal bundle $N_{C_1/X}= \OO_{\PP^1}^{k-e_0+1} \oplus \OO_{\PP^1}(1)^{e_0-2}$. The normal bundle $N_{C_1/X}$ has nonnegative degree, hence is globally generated. Consequently, we may assume $C_1$ passes through a general point $q$ of $X$.  Let $C_2$ be a degree $k+1$ rational curve through $q$ with balanced normal bundle $N_{C_2/X} = \OO_{\PP^1}(1)^{k-1}$. Let $C = C_1 \cup C_2$. Since the normal bundle of $C_2$ is ample, we can assume that the tangent direction $T_q C_2$ is general with respect to $C_1$. Consequently, by Lemma \ref{restrictedNormalBundle} (1), $N_{C/X|_{C_1}}=  \OO_{\PP^1}^{k-e_0} \oplus \OO_{\PP^1}(1)^{e_0-1}$ is balanced and $N_{C/X|_{C_2}} = \OO_{\PP^1}(1)^{k-2} \oplus \OO_{\PP^1}(2)$. By Lemma \ref{restrictedNormalBundle} (2), $C$ can be smoothed to a curve with normal bundle that is either balanced or has the form $\OO(1)^{2k-e+1} \oplus \OO(2)^{e-k-3} \oplus \OO(3)$. This concludes the proof.
\end{proof}

We next  build on Proposition \ref{glue2k} to analyze complete intersections of quadrics.

\begin{proposition}\label{gluequadrics}
Let $n > 2k$ and let $X$ be a general complete intersection of $k$ quadrics in $\PP^n$. Then $X$ contains rational curves of every degree $1 \leq e \leq n-1$ with balanced normal bundle. Furthermore, $X$ contains a rational curve of degree $n$ with $i(N_{C/X}) \leq 2$ and $\delta(N_{C/X}) \leq 1$. 
\end{proposition}

\begin{proof}
The proof is by induction on $n$. To check the base case assume that $n=2k+1$.  If $1 \leq e \leq k+2$, then the Proposition is implied by Corollary \ref{rPlus2Cor}. We may assume $k+3 \leq e \leq 2k$. Let $X_0$ be a general intersection of $k$ quadric hypersurfaces in $\PP^{2k}$.  If $X_0$ contains rational curves of degree $k+3 \leq e \leq 2k$ with balanced normal bundle, then by Proposition \ref{hyperplaneProp} so does the general $X$. If not, by Proposition \ref{glue2k}, we may assume that $X_0$ contains rational curves with normal bundle $\OO_{\PP^1}(1)^{2k-e+1} \oplus \OO_{\PP^1}(2)^{e-k-3} \oplus \OO_{\PP^1}(3)$. By Proposition \ref{hyperplaneProp}, every extension of $$0 \rightarrow N_{C/X_0} \rightarrow N_{C/X} \rightarrow \OO_{\PP^1}(e) \rightarrow 0$$ occurs for some $X$ containing $X_0$. Since $e \geq 2k-e+3$, by Lemma \ref{lem-generalextension} the general such extension is balanced.

To complete the proof of the base case, it remains to exhibit a rational curve of degree $2k+1$ with $\delta(N_{C/X}) \leq 1$. We can find curves $C_1$ of degree $k+1$ with normal bundle $\OO_{\PP^1}(2)^k$ and $C_2$ of degree $k$ with normal bundle $\OO_{\PP^1}(1)^2 \oplus \OO_{\PP^1}(2)^{k-2}$ both passing through a general point $q$ of $X$. Since both of their normal bundles are ample, we may assume that $T_q C_i$ is general with respect to $C_j$ for $i \neq j$. Let $C = C_1 \cup C_2$. By Lemma \ref{restrictedNormalBundle} (1), $N_{C/X|_{C_1}}= \OO_{\PP^1}(1) \oplus \OO_{\PP^1}(2)^{k-1}$ and $N_{C/X|_{C_2}} = \OO_{\PP^1}(2)^{k-1} \oplus \OO_{\PP^1}(3)$. By Lemma \ref{restrictedNormalBundle} (2), $C$ can be smoothed and the normal bundle is either balanced or has the form $\OO_{\PP^1}(3)\oplus \OO_{\PP^1}(4)^{k-2} \oplus \OO_{\PP^1}(5)$. This concludes the base case $n=2k+1$.

Assume that the proposition holds up to $n=m-1\geq 2k+1$. Let $X_0$ be a general complete intersection of $k$ quadrics in $\PP^{m-1}$ and let $X$ be a general complete intersection of $k$ quadrics in $\PP^m$ that contains $X_0$ as a hyperplane section. By the inductive hypothesis, $X_0$ contains rational normal curves with balanced normal bundles of every degree $1 \leq e \leq m-2$. By Proposition \ref{hyperplaneProp}, $X$ also contains rational normal curves of these degrees with balanced normal bundle. Furthermore, by Corollary \ref{cor-balanced2} and Proposition \ref{hyperplaneProp}, we may assume that $k \geq 2$. For $e=m-1$, $X_0$ contains a rational normal curve with $\delta(N_{C/X})\leq 1$. Using the division algorithm write $(m-1)(m-2k)-2= (m-k-2)q + r$ with $0 \leq r < m-k-2$. Then the normal bundle of the general curve of degree $m-1$ on $X_0$, if not balanced,  is $\OO_{\PP^1}(q-1) \oplus \OO_{\PP^1}(q)^{m-k-r-4} \oplus \OO_{\PP^1}(q+1)^{r+1}$. By Proposition \ref{hyperplaneProp}, the general extension of $\OO_{\PP^1}(m-1)$ by $N_{C/X_0}$ occurs as $N_{C/X}$. Since $k \geq 2$, we  have
$$q \leq \frac{(m-1)(m-2k)-2}{m-k-2} = (m-1) - \frac{(k-2)(m-1) + 2}{m-k-2} < m-1.$$
Consequently, $q+1 \leq m-1$ and  by Lemma \ref{lem-generalextension} the general extension is balanced. 

Finally, we discuss the case $e=m$ in $\PP^m$. Let $C_1$ and $C_2$ be rational curves of degrees $m-k-1$ and $k+1$, respectively, passing through a general point $q$ of $X \subset \PP^m$ and having balanced normal bundles in $X$. The normal bundle of $C_1$ is  $\OO_{\PP^1}(m-2k)^2 \oplus \OO_{\PP^1}(m-2k+1)^{m-k-3}$ and the normal bundle of $C_2$ is balanced and ample. Hence, $C_1$ and $C_2$ will have general tangent directions through $q$. We conclude that $N_{C/X|_{C_1}} = \OO_{\PP^1}(m-2k) \oplus \OO_{\PP^1}(m-2k+1)^{m-k-2}$ and $N_{C/X|_{C_2}}$ is balanced. Hence, $C$ smooths to a rational curve with normal bundle having indentation at most 1. This concludes the induction step and the proof of the proposition.
\end{proof}

\begin{proof}[Proof of Theorem \ref{Thm-mainbalanced}]
For $d_1 = \dots = d_k = 2$, the theorem follows from Proposition \ref{gluequadrics}. If none of the $d_i$ are 2, then the theorem follows from  Corollary \ref{no2s}. We may assume $d_1 = \dots =d_j =2 < d_{j+1} \leq \cdots \leq d_k$, for $1 \leq j < k$. Note that in this case, $n > 2k$. First, we prove the theorem for $1 \leq e \leq n-1$. By Proposition \ref{gluequadrics} a general complete intersection of $j$ quadric hypersurfaces contains a rational normal curve $R_e$ with balanced normal bundle. Applying Proposition \ref{passingtohypersurface} repeatedly, we conclude that $X$ also contains rational normal curves $R_e$ with balanced normal bundle. The same argument holds true when $e=n$ and $j = 1$ by Corollary \ref{cor-balanced2}. There remains to show the case $e=n$ and $1 < j < k$.  By Proposition \ref{gluequadrics}, there exists a rational normal curve $R_n$ on the general complete intersection of $j$ quadric hypersurfaces with indentation at most $1$. Since $d_{j+1} \geq 3$, we have $$\left \lceil \frac{n(n-2j-d_{j+1} +1)-2}{n-j-2} \right\rceil \leq \left\lfloor \frac{n(n-2j +1)-2}{n-j-1} \right\rfloor -1 .$$ By Lemma \ref{lem-generalkernel} and Proposition \ref{passingtohypersurface}, the normal bundle of $R_e$ in the intersection of the first $j+1$ hypersurfaces is balanced. The theorem  follows by repeatedly applying Proposition \ref{passingtohypersurface}. 
\end{proof}

As a consequence of Theorem \ref{Thm-mainbalanced}, we strengthen Corollary \ref{cino2s} to include all Fano complete intersections. We recover the theorem of Z. Tian \cite{Tian},  Q. Chen and Y. Zhu \cite{ChenZhu} with sharp degree bounds.

\begin{theorem}
Let $X \subset \PP^n$ be a general Fano complete intersection of type $(d_1, \dots, d_k)$. Let $d = \sum_{i=1}^k d_i$. Then $X$ contains a very free rational curve of every degree $e \geq m = \lceil \frac{n-k+1}{n-d+1} \rceil$.
\end{theorem}
\begin{proof}
We first address the case $e=m$.  If all $d_i$ are equal to $2$, we have $m \leq k+1$. Hence, there exists rational curves of degree $m$ with balanced normal bundle by Corollary \ref{rPlus2Cor}. Otherwise, $m \leq n$ and there exists rational curves of degree $m$ with balanced normal bundle. Since the degree of the normal bundle is at least the rank, these normal bundles are ample. 

To obtain curves of larger degree $e$ with ample normal bundle, we can attach a general line $\ell$ to a general curve $C_0$ of degree $e-1$ with ample normal bundle. Let $C = \ell \cup C_0$. The normal bundle $N_{\ell/X}$ is either globally generated or when $d=n$ it has a unique factor of $\OO_{\PP^1}(-1)$. In either case, when we attach $C_0$ with general tangent direction, Lemma \ref{restrictedNormalBundle} (1) implies that $N_{C/X}|_{\ell}$ is globally generated. Since the normal bundle of $C_0$ is ample, Lemma  \ref{restrictedNormalBundle} (1)  implies that every summand in $N_{C/X}|_{C_0}$ has positive degree. By Lemma \ref{restrictedNormalBundle} (2), $C$ smooths to a rational curve where all the summands of the normal bundle have positive degree. The theorem follows.
\end{proof}

\bibliographystyle{plain}

\end{document}